\documentclass[twoside, 11pt]{article}

\usepackage{amsfonts}
\usepackage{amsthm}
\usepackage{amsmath}
\usepackage{amssymb}
\usepackage[utf8]{inputenc}
\usepackage{graphicx}
\usepackage{fancyhdr}
\usepackage{hyperref}
\usepackage{geometry}
\usepackage{enumerate}
\usepackage{xcolor}
\usepackage{authblk}
\usepackage{subcaption}
\usepackage{cite}

\newcommand{\R}{{\mathbb R}}

\newcommand{\K}{\mathbb{K}}
\newcommand{\N}{{\mathbb N}}
\renewcommand{\k}{{\bf k}}
\newcommand{\ke}{{\bf e}_1}
\newcommand{\ked}{{\bf e}_2}

\newcommand{\spk}{\operatorname{Sup}}
\newcommand{\diam}{\operatorname{diam}}

\newcommand{\calI}{\mathcal{I}}
\newcommand{\calP}{\mathcal{P}}
\newcommand{\calV}{\mathcal{V}}

\newcommand{\frakD}{\mathfrak{D}}
\newcommand{\frakI}{\mathfrak{I}}
\newcommand{\frakP}{\mathfrak{P}}

\theoremstyle{plain}
\newtheorem{theorem}{Theorem}[section]
\newtheorem{corollary}[theorem]{Corollary}
\newtheorem{lemma}[theorem]{Lemma}

\theoremstyle{definition}
\newtheorem{definition}[theorem]{Definition}
\newtheorem{example}[theorem]{Example}

\theoremstyle{remark}
\newtheorem{remark}[theorem]{Remark}

\definecolor{org1}{RGB}{67, 132, 252}
\hypersetup{colorlinks=true, allcolors=blue}

\begin{document}
\title{On the Notion of a Function of Bounded Variation and of Riemann-Stieltjes Integral with Strong Partitions on Hyperbolic Intervals}
\author[$\dagger$]{G. Y. Tellez-Sanchez}
\author[$\star$]{J. Bory-Reyes}
\date{\today}
  
\affil[$\dagger$]{
\footnotesize
Escuela Superior de Fisica y Matem\'aticas

Instituto Polit\'ecnico Nacional

Edif. 9, 1er piso, U.P. Adolfo L\'opez Mateos

07338, Mexico City,  MEXICO

gtellez.wolf@gmail.com

ORCID: 0000-0002-6387-5336
\vspace{3mm}
}
 
\affil[$\star$]{
\footnotesize
Escuela Superior de Ingenieria Mec\'anica y El\'ectrica

Instituto Polit\'ecnico Nacional

Edif. 5, 3er piso, U.P. Adolfo L\'opez Mateos

07338, Mexico City, MEXICO

juanboryreyes@yahoo.com

ORCID: 0000-0002-7004-1794
}

\maketitle

\begin{abstract}
In this paper we provided a classification for partitions of intervals on the hyperbolic plane. Given a partition, to be named strong, we define a notion of a hyperbolic-valued functions of bounded variation and a kind of Riemann-Stieltjes integral. A condition relating to both concepts appears to be natural for the existence of the integral, as it occurs in real analysis.
\end{abstract}


\section{Introduction}
Hyperbolic numbers, also called split-complex numbers, dates back to 1848 when James Cockle revealed his tessarines in \cite{c1848}. In the beginning, between 1935-1941 a theory of hyperbolic-valued functions in the hypercomplex analysis setting was deeply investigated and extensively continued in many respect, as an appropriate generalization to the real analysis, see for instance \cite{als2017, g1997, g2001, CST2018, ks2017, ks2016, kgs2020, lssv2015, lps2014, s1995, ssk2020, t2018, tb2017, tb2021, tb2019, vd1935, v1938}.

Using the idempotent representation of hyperbolic number, introduced early in \cite{vd1935}, and also know as the light cone basis, we can identify the hyperbolic numbers plane with the direct product of real numbers with itself. The significance of this fact is revealed from the possibility to extends the total order in real numbers to hyperbolic numbers in a partial order sense, doing use of the point-wise well order in the Euclidean plane \cite{bbls2016}. One of the appealing aspects of the idempotent representation is that it readily lends itself to make our calculations easier.

The remarkable feature of partitions over hyperbolic intervals have allowed to bring a certain classification of Cantor type sets \cite{bbls2016, tb2017, t2018}, to introduce hyperbolic-valued probabilities \cite{als2017} and subsequently its application to the chaos game algorithm for hyperbolic numbers appeared in \cite{tb2021}. In \cite{bps2019} was introduced a Cauchy type integral in the bicomplex setting, and afterward a notion of integration of hyperbolic-valued functions was presented in \cite{el2022}. A concept of hyperbolic-valued measure can be found in \cite{ks2017, CST2018}. Proceeding further in this direction, a first presentation of a Riemann-Stieltjes integral for hyperbolic-valued functions appeared in \cite{tb2021b}, meanwhile \cite{gm2022} contains an extension to the bicomplex valued functions framework. State that Cauchy-like integral formula holds for functions of a hyperbolic variable has been recently shown, see \cite{b2013, cz2012, l2007}. 

The paper is organized as follows. After this brief introduction, in Section \ref{hvFT} we recall basic concepts and facts on hyperbolic-valued function theory. Section \ref{scHP} introduces a classification for hyperbolic-valued partitions where strong partitions are outstanding because it lets to define a hyperbolic-valued function of bounded variation and a Riemann-Stieltjes integral. In Section \ref{scHFBV}, hyperbolic-valued functions of bounded variations are defined and one of the main results is presented. It says that all discontinuities of functions of bounded variations on the hyperbolic plane can be located in parallel lines to idempotent axes. As a consequence, the set of all points of discontinuity of a function of bounded variation has Lebesgue measure equal to zero, like it occurs in real analysis \cite{a1974}. 

In final section we develop the theory of Riemann-Stieltjes integral over hyperbolic-valued functions, where the notion of strong partitions have been taken into consideration. Necessary and sufficient conditions of the ensuring the existence of integral are given, which also cover the relation of hyperbolic-valued functions of bounded variation with the Riemann-Stieltjes integral.

\section{Hyperbolic-valued function theory}\label{hvFT}
\subsection{Hyperbolic Numbers}
The ring of hyperbolic numbers is generated by real numbers and one imaginary unit $\k \not\in \R$ with the property that $\k^2 = 1$, which we can write as
\[\K := \R[\k] = \{ t + s\k \ |\ t, s \in \R \}.\]
In a similar way like complex numbers, $\K$ can be identified with the Euclidean plane and it is customary to call hyperbolic numbers plane.

Complex and hyperbolic numbers are different number systems because the first has the structure of field meanwhile the second is a commutative unit ring with zero divisors.

There are two important zero divisors in $\K$:
\[\ke := \frac{1+\k}{2}, \quad \ked := \frac{1 - \k}{2}.\]
They are idempotent elements and the product of both numbers is equal to zero.

Although $\ke$ and $\ked$ are zero divisors, every element $\alpha \in \K$ is represented in an unique real linear combination of these idempotent elements.
\[\alpha = t + s\k \Rightarrow \alpha = (t + s)\ke + (t - s)\ked. \]
\[\alpha = a_1\ke + a_2\ked \Rightarrow \alpha = \frac{1}{2}(a_1 + a_2) + \frac{1}{2}(a_1 - a_2)\k. \]
When $\alpha$ is taken as $a_1\ke + a_2\ked$, it will be said that $\alpha$ is in idempotent representation.

Using idempotent representation, the hyperbolic numbers plane is identify with $\R\ke + \R\ked$ and there exists a ring isomorphism between $\K$ and the direct product $\R \otimes \R$, see for instance \cite{vd1935, s1995}. The real line is endowed in $\K$ by the mapping $x \mapsto x\ke + x\ked$. Also, the idempotent projections of a subset $A \subset \K$ are given by
\[A_{\ke} := \{ a \in \R \ |\ \exists b\in \R,\ a\ke + b\ked \in A \}, \]
\[A_{\ked} := \{ b \in \R \ |\ \exists a \in \R,\ a\ke + b\ked \in A \}. \]
\subsection{Partial order}
On the realization of hypercomplex numbers, a partial order relation is defined in \cite{lssv2015}. For two hyperbolic numbers $\alpha= a_1\ke + a_2\ked$ and $\beta= b_1\ke + b_2\ked$ are related by $\preceq$, if
\[a_1 \leq b_1 \ \land \ a_2 \leq b_2. \]
One can check that this relation is reflexive, transitive and antisymmetric and represents a good generalization for the total order in the real numbers.

Strict order for hyperbolic numbers is defined by the rule
\[ \alpha \prec \beta \quad \Leftrightarrow \quad a_1 < b_1 \ \land \ a_2 < b_2.\]
Here, strict order is handled in a different way like that defined at \cite{kgs2020}, where it is considered only one of the two cases: $a_1 = b_1$ or $a_2 = b_2$. 

Closed and open intervals are defined relative to the partial order as the sets
\begin{displaymath}
\mbox{Closed: }[\alpha, \beta]_{\k} := \{ \xi \in \K \ |\ \alpha \preceq \xi \preceq \beta \},
\end{displaymath}
\begin{displaymath}
\mbox{Open: } (\alpha, \beta)_{\k} := \{ \xi \in \K \ |\ \alpha \prec \xi \prec \beta \}.
\end{displaymath}

The length of a hyperbolic interval $\frakI = [\alpha, \beta]_{\k}$, is by definition, the number:
\begin{equation}\label{length}
\lambda_{\k}(\frakI) := \beta - \alpha.
\end{equation}
A closed interval $[\alpha, \beta]_{\k}$ is said to be degenerated when $a_1 = b_1$ or $a_2 = b_2$. This type of intervals are identified with a line segment bounded by $\alpha$ and $\beta$.
\subsection{Holomorphic hyperbolic-valued functions}\label{hypValuedFunctions}
The conjugate of a hyperbolic number $\alpha= t + s\k$ is another hyperbolic number denoted by $\alpha^{\dagger} := t - s\k$. But when $\alpha$ is in idempotent representation $\alpha = a_1\ke + a_2\ked$, its conjugate is getting by $\alpha^{\dagger} := a_2\ke + a_1\ked$.

We will consider the hyperbolic-valued module of a hyperbolic number $\alpha$ defined by
\[ |\alpha|_{k}:= |a_1|\ke + |a_2|\ked. \]
Since $| \alpha |_{\k} = \sqrt{\alpha^2}$, it is reasonable that the behaviour of hyperbolic module should in some sense approximate the behaviour of real module. Moreover, the hyperbolic-valued module generates a hyperbolic-valued metric (see \cite{kgs2020, ks2016, lssv2015}) given by 
$$D_{\k} = |\alpha - \beta |_{\k}$$ 
with the properties
\begin{enumerate}
\item $D_{\k}(\alpha, \beta) = 0$, if and only if $\alpha = \beta$.
\item $D_{\k}(\alpha, \beta) = D_{\k}(\beta, \alpha)$.
\item $D_{\k}(\alpha, \gamma) \preceq D_{\k}(\alpha, \beta) + D_{\k}(\beta, \gamma)$.
\end{enumerate}
We call $D_{\k}$ the usual hyperbolic metric over $\K$.

With the structure of hyperbolic metric space, the concept of continuous function was introduced in \cite{kgs2020}.
\begin{definition}
Let $F$ be a hyperbolic-valued function over $\Omega \subset \K$. We say that $F$ is a $\K$-continuous function in $\xi \in \Omega$, when given a hyperbolic positive number $\epsilon$ there exists another hyperbolic positive number $\delta$ such that $D_{\k}(\xi, \zeta) \prec \delta$, with the condition 
$$D_{\k}(F(\xi), F(\zeta)) \prec \epsilon.$$
\end{definition}
We will need the requirement on hyperbolic-valued functions $F= F_1 + F_2\k: \Omega \rightarrow \K$ that it have real-valued components $F_1, F_2$ of one variable, which is guaranteed on the idempotent representation if, for instance, $F$ satisfies the Cauchy-Riemann type system 
\begin{equation}\label{crsystem}
\frac{\partial F_1(\xi)}{\partial x_1}= \frac{\partial F_1(\xi)}{\partial x_2}, \quad \land \quad \frac{\partial F_1(\xi)}{\partial x_2} = \frac{\partial F_2(\xi)}{\partial x_1},
\end{equation}
for $\xi= x_{1}\ke + x_{2}\ked \in \Omega$. 

When $F = F(\xi) = F_1(x_{1})\ke + F_2(x_{2})\ked$ fulfill the system \ref{crsystem}, we obtain what is known as $\K$-derivable or holomorphic hyperbolic-valued function, with derivative
\[ F'(\xi) := dF_1(x_{1})\ke + dF_2(x_{2})\ked, \]
where $dF_1, dF_2$ are the real derivatives of $F_1$ and $F_2$ respectively. For a deeper discussion of this result and properties of holomorphic hyperbolic-valued functions we refer the reader to \cite{s1995, vd1935, v1938}.

A holomorphic function defined over an open connected set $\Omega$, is extended to $\Omega_{\ke}\ke + \Omega_{\ked}\ked$, to be the open interval $(\omega_1, \omega_2)_{\k}$ with
\[\omega_{1} = \inf \Omega_{\ke}\ke + \inf \Omega_{\ked}\ked, \quad \omega_2 = \sup \Omega_{\ke} \ke + \sup \Omega_{\ked} \ked. \]
This result appears in \cite{vd1935}. 

\section{Classification of Hyperbolic Partitions}\label{scHP}
A non-degenerated hyperbolic interval can be consider like a rectangle in the Euclidean plane, therefore it is possible to divide the interval $[\alpha, \beta]_{\k}$ by a classical partition $P = P_1 \times P_2$, where $P_1 \subset [a_1, b_1]$ and $P_2 \subset [a_2, b_2]$ are partitions from its respective intervals. A shortcoming of this division is the inability to obtain the length of $[\alpha, \beta]_{\k}$ equals the sum of the length of the union of sub-rectangles of $P$, see \cite{bbls2016}.
\begin{example}\label{ex1}
Let the interval $\frakI = [0, 1]_{\k}$ divided by the following nine sub-intervals: 
\[\small
\arraycolsep=1.4pt\def\arraystretch{2.2}
\begin{array}{ccc}
\frakI_1 = \left[0, \frac{1}{3}\right]_{\k}, & \frakI_2 = \left[\frac{1}{3}, \frac{2}{3} \right]_{\k}, & \frakI_3 = \left[\frac{2}{3}, 1 \right]_{\k}, \\
\frakI_4 = \left[\frac{1}{3}\ke, \frac{2}{3}\ke + \frac{1}{3}\ked \right]_{\k}, & \frakI_5 = \left[\frac{2}{3}\ke, 1\ke + \frac{1}{3}\ked \right]_{\k}, & \frakI_6 = \left[\frac{2}{3}\ke + \frac{1}{3}\ked, 1\ke + \frac{2}{3}\ked \right]_{\k}, \\
\frakI_7 = \left[\frac{1}{3}\ked, \frac{1}{3}\ke + \frac{2}{3}\ked \right]_{\k}, & \frakI_8 = \left[\frac{2}{3}\ked, \frac{1}{3}\ke + 1\ked \right]_{\k}, & \frakI_9 = \left[\frac{1}{3}\ke + \frac{2}{3}\ked, \frac{2}{3}\ke + 1\ked \right]_{\k}.
\end{array}
\]
By the definition (\ref{length}), the sum of the lengths of  all sub-intervals $\frakI_j$, with $j \in \{1, ..., 9\}$, is not equal to the length of $[0, 1]_{\k}$.
\[\lambda_{k}\left( \frakI \right) = 1 \neq 3 = \sum_{j = 1}^{9} \lambda_{\k} \left( \frakI_j \right). \]
\end{example}
We now indicate how that difficulty can be circumvented.
\begin{definition}\label{df:RegPart}
Any finite collection $S_1, ..., S_n$ of sub-intervals into which $[\alpha, \beta]_{\k}$ is divided: 
\[ [\alpha, \beta]_{\k} = \bigcup_{j = 1}^{n}S_{j},\]
will be called a regular partition if 
\[\mu_{\R^{2}} \left( [\alpha, \beta]_{\k} \right) = \sum_{j = 1}^{n} \mu_{\R^{2}} \left( S_j \right),\]
where $\mu_{\R^{2}}$ denotes the Lebesgue measure in the Euclidean plane.
\end{definition}

\begin{remark}
Definition \ref{df:RegPart} is not restricted to the case of sub-rectangles generated by the Cartesian product of partitions from real intervals like in Example \ref{ex1}.
\end{remark}
Let us introduce the following natural definition of hyperbolic interval partition.
\begin{definition}\label{waekpartition}
Any collection $\calI$ of sub-intervals into which $[\alpha, \beta]_{\k}$ is divided:
\[[\alpha, \beta]_{\k} = \bigcup_{I \in \calI}I\]
is said to be a weak partition if 
\[\lambda_{\k}([\alpha, \beta]_{\k}) = \sum_{I \in \calI} \lambda_{\k} (I).\]
\end{definition} 
Weak partitions allows that the set of sub-intervals could have empty intersections like Fig. \ref{fg:wp} shows.
\begin{figure}[ht]
\begin{subfigure}{.5\textwidth}
\centering
\includegraphics[scale=1]{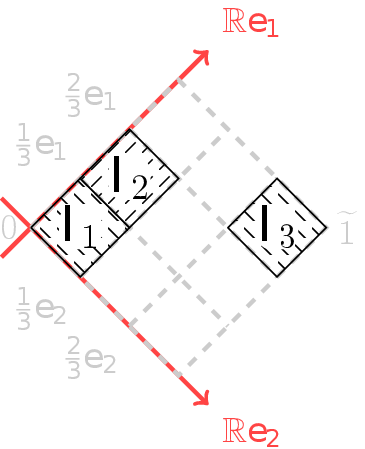}
\caption{ }
\label{fg:wp1}
\end{subfigure}
\begin{subfigure}{.5\textwidth}
\centering
\includegraphics[scale=1]{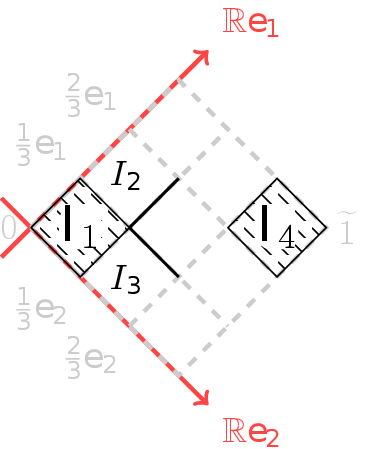}
\caption{ }
\label{fg:wp2}
\end{subfigure}
\caption{Example of weak partition}
\label{fg:wp}
\end{figure}

Figure \ref{fg:wp1} shows three sub-intervals from $[0, 1]_{\k}$ all with length equal to $\displaystyle\frac{1}{3}$. Therefore
\[\lambda_{\k}(I_1) + \lambda_{\k}(I_2) + \lambda_{\k}(I_3) = 1.\]
In the other hand, Fig. \ref{fg:wp2} shows four intervals, where $I_1$ and $I_4$ have length equal to $\frac{1}{3}$, meanwhile for the two remaining intervals we have
\[\lambda_{\k}(I_2) = \frac{1}{3}\ke \text{ and } \lambda_{\k}(I_3) = \displaystyle\frac{1}{3}\ked.\]
So, the sum of lengths of the four intervals is again equal to $1$.
\begin{remark}
In \cite{bbls2016}, to avoid the disjoint intervals issue, a total order over the points that generate a partition is required.
\end{remark}

\begin{definition} \label{strongpartition}
Let $\frakP = \{\rho_0, ..., \rho_n\}$ be a finite collection of points in the interval $[\alpha, \beta]_{\k}$ such that $\rho_s \neq \rho_t$ when $s \neq t$. We say that $\frakP$ is a strong partition, if both conditions are fulfill
\begin{enumerate}[1)- \quad]
\item $\frakP$ is a chain on $\K$.
\item\label{strongpartition2} $\rho_0 = \alpha$, $\rho_n = \beta$ and
\[ \rho_0 \preceq \rho_1 \preceq ... \preceq \rho_n .\]
\end{enumerate}
\end{definition}

Definition \ref{strongpartition} is less restrictive than that given in \cite{bbls2016}, however the term ``strong" makes reference to the comparative with Def. \ref{waekpartition}.

The equality in (\ref{strongpartition}-\ref{strongpartition2}) is not allowed by the authors in \cite{bbls2016}, therefore their definition is restricted to no zero divisor elements. However, this restriction in Def. \ref{strongpartition} does not alter the proof given in \cite{bbls2016} for the next result.

\begin{theorem}
If $\frakP$ is a strong partition of $[\alpha, \beta]_{\k}$, then
\[ \sum_{j = 1}^{n} \lambda_{\k}([\rho_{j-1}, \rho_{j}]_{\k}) = \lambda_{\k} ([\alpha, \beta]_{\k}).\]
\end{theorem}

Fig \ref{fg:sp1} shows how is the design of an strong partition where zero divisor elements are included. While Fig \ref{fg:sp2} is the extension of the real partition with diameter equal to $\frac{1}{3}$ to hyperbolic plane.

\begin{figure}[ht]
\begin{subfigure}{.5\textwidth}
\centering
\includegraphics[scale=1]{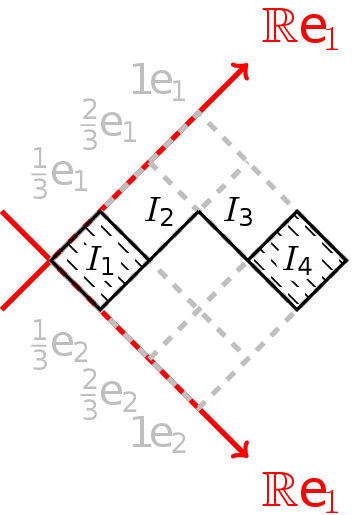}
\caption{ }
\label{fg:sp1}
\end{subfigure}
\begin{subfigure}{.5\textwidth}
\centering
\includegraphics[scale=1]{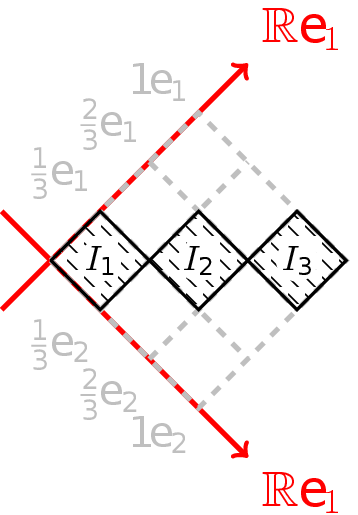}
\caption{ }
\label{fg:sp2}
\end{subfigure}
\caption{Example of strong partition}
\label{fg:sp}
\end{figure}

Projections of one partition $\frakP = \{ \rho_0, \rho_2, ..., \rho_{n} \}$ over the interval $[\alpha, \beta]_{\k}$, makes two partitions of the real intervals $[a_1, b_1]$ and $[a_2, b_2]$.
\begin{equation}\label{eq2}
\begin{split}
\frakP_{\ke} & = \{ p_{0, 1}, p_{2, 1}, ..., p_{n, 1} \} \\
\frakP_{\ked} & = \{p_{0, 2}, p_{2, 2}, ..., p_{n, 2} \},
\end{split}
\end{equation}
where $\rho_{j} = p_{j, 1}\ke + p_{j, 2}\ked$ for every $j \in \{1, 2, ..., n \}$.

Now, if the partitions $P = \{p_0, ..., p_s\} \subset [a_1, b_1]$ and $Q = \{q_0, ..., q_t\} \subset [a_2, b_2]$ are given, then it is possible to build one hyperbolic partition of $[\alpha, \beta]_{\k}$.

By definition of partition, $\alpha = p_{0}\ke + q_0\ked$ and $\beta = p_{s}\ke + q_{t}\ked$. These are renamed as $\alpha = \rho_{0, 0}$ and $\beta = \rho_{s, t}$.

The general process to get a hyperbolic point is taking points $p_{s_j} \in P$ and $q_{t_j} \in Q$ with $p_{s_j - 1} \leq p_{s_j} \leq p_{s}$ and $q_{t_j - 1} \leq q_{t_j} \leq q_{t}$, but if $p_{s_j} = p_{s_j - 1}$, then $q_{t_j} \in Q \setminus \{q_{t_j - 1}\}$, in a similar way if $q_{t_j} = q_{t_j -1}$, then $p_{s_j} \in P \setminus \{p_{s_j - 1}\}$. It defines the point $\rho_{s_j, t_j} = p_{s_j}\ke + q_{t_j}\ked$.

Previous step only can be repeated in a maximum of $s + t$ times and it finishes when $\rho_{s_j, t_j} = \rho_{s, t}$.

This procedure generates a strong partition 
\begin{equation}\label{mkStrongPart}
\frakP = \{\rho_{1, 1}, \rho_{s_1, t_1}, ..., \rho_{s, t} \},
\end{equation}
although $\frakP$ is not unique. Figure \ref{realp} shows some examples.

\begin{figure}[ht]
\begin{subfigure}{.5\textwidth}
\centering
\includegraphics[scale=0.7]{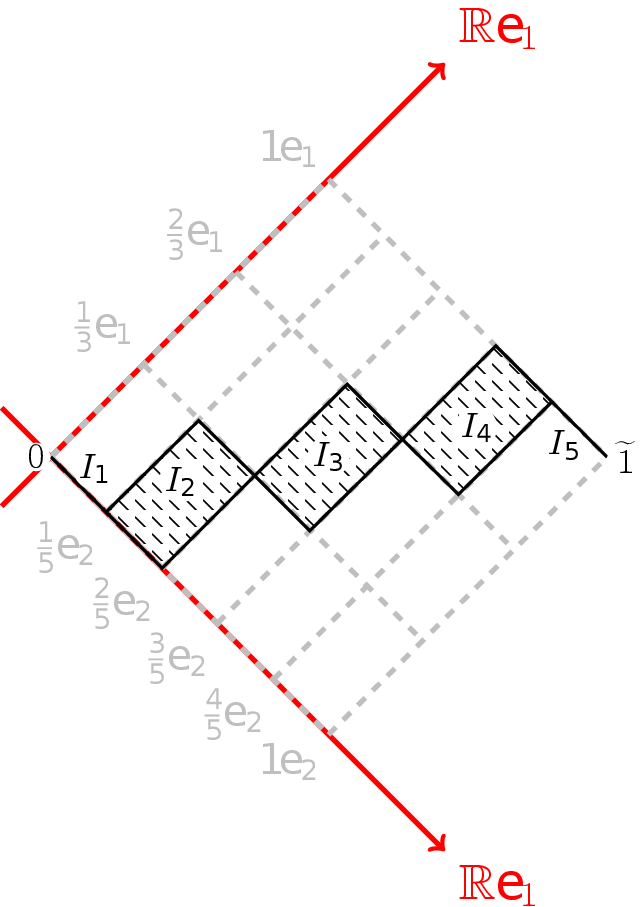}
\caption{ }
\label{realp1}
\end{subfigure}
\begin{subfigure}{.5\textwidth}
\centering
\includegraphics[scale=0.7]{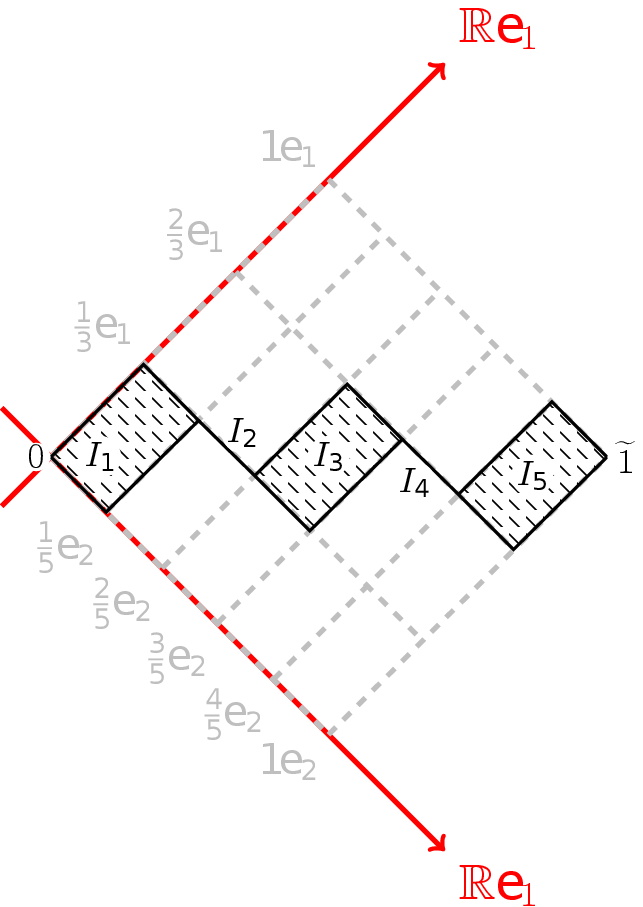}
\caption{ }
\label{realp2}
\end{subfigure}
\caption{Strong partition generated by $\displaystyle P_1 = \left\{0, \frac{1}{3}, \frac{2}{3}, 1\right\}$, $\displaystyle P_2 = \left\{0, \frac{1}{5}, \frac{2}{5}, \frac{3}{5}, \frac{4}{5}, 1\right\}$. Both are partitions on the real interval $[0, 1]$.}
\label{realp}
\end{figure}

\section{Hyperbolic Functions of Bounded Variation}\label{scHFBV}
Let $[\alpha, \beta]_{\k} \subset \K$ and $F: [\alpha,  \beta]_{\k} \rightarrow \K$ a hyperbolic-valued function with idempotent representation $F = F_{1}\ke + F_{2}\ked$. 

If $\frakP = \{\rho_0, \rho_1, ..., \rho_{n}\}$ is a strong partition of $[\alpha, \beta]_{\k}$, then we can consider the quantity
\begin{equation}\label{eq1}
\begin{split}
\sum_{j = 0}^{n-1} & |F(\rho_{j+1}) - F(\rho_{j})|_{\k} = \\
& \sum_{j=0}^{n-1} \left( |F_1(\rho_{j + 1}) - F_1(\rho_{j})|\ke + |F_2(\rho_{j + 1}) - F_2(\rho_{j})|\ked \right) =  \\
& \left( \sum_{j=0}^{n-1} |F_1(\rho_{j + 1}) - F_1(\rho_{j})| \right)\ke + \left( \sum_{j=0}^{n-1} |F_2(\rho_{j + 1}) - F_2(\rho_{j})| \right)\ked  = \\
& \left( \sum_{j=0}^{n-1} \Delta_{\frakP, j}F_1 \right) \ke + \left( \sum_{j=0}^{n-1} \Delta_{\frakP, j}F_2 \right)\ked.
\end{split}
\end{equation}

Let $\calP([\alpha, \beta]_{\k})$ denotes the family of all strong partitions for $[\alpha, \beta]_{\k}$ and it defines the set
{\small
\[ \sum_{\calP([\alpha, \beta]_{\k})} (F) := \left\{ \left( \sum_{j=0}^{n_\frakP-1} \Delta_{\frakP, j}F_1 \right) \ke + \left( \sum_{j=0}^{n_\frakP-1} \Delta_{\frakP, j}F_2 \right)\ked \  \big| \ \frakP \in \calP([\alpha, \beta]_{\k})  \right\}. \]
}
\begin{definition}\label{d4}
A hyperbolic-valued function $F: [\alpha, \beta]_{\k} \rightarrow \K$ is said to have bounded variation, if 
\[\spk\left( \sum_{\calP([\alpha, \beta]_{\k})} (F) \right) \prec \infty.\]
\end{definition}

The supremum is taken like in references \cite{tb2017, t2018, lps2014, ssk2020}. On a nonempty set $A \subset \K$,
\[\spk(A) := \sup(A_{\ke})\ke + \sup(A_{\ked})\ked.\]

So, to say that a function has bounded variation is equivalent to require that
\[\left( \sum_{\calP([\alpha, \beta]_{\k})} (F) \right)_{\ke} \text{ and } \quad  \left( \sum_{\calP([\alpha, \beta]_{\k})} (F)\right)_{\ked} \]
are bounded sets of real numbers.

\begin{definition}\label{d5}
The total variation of a function of bounded variation $F: [\alpha, \beta]_{\k} \rightarrow \K$ is given by
\[\calV_{[\alpha, \beta]_{\k}}(F) := \spk\left( \sum_{\calP([\alpha, \beta]_{\k})} (F) \right).\]
\end{definition}
Therefore,
{\small
\[
\begin{split}
\calV_{[\alpha, \beta]_{\k}}(F) & = \sup_{\frakP \in \calP([\alpha, \beta]_{\k})}  \left( \sum_{j=0}^{n_\frakP-1} \Delta_{\frakP, j}F_1 \right)\ke + \sup_{\frakP \in \calP([\alpha, \beta]_{\k})}  \left( \sum_{j=0}^{n_\frakP-1} \Delta_{\frakP, j}F_2 \right)\ked.
\end{split}
\]}

Components of a function that satisfies Def. \ref{d4} do not necessarily have Vitali variation finite (see \cite{v1908, f1910, ca1933, ad2015} for definition of Vitali variation). A function of finite Vitali variation takes account the four vertex from every sub-interval (or sub-square) in the real plane generated in a regular partition. Meanwhile the hyperbolic version only takes two extremes (or two vertex) from every interval in a strong partitions, which is a proper subset of the regular partition.

Special case to consider are functions when every component depends only on the respective component of the variable. So, let $F:[\alpha, \beta]_{\k} \rightarrow \K$ be a function such that $F\left( \xi \right) = F_1(x_1)\ke + F_2(x_2)\ked$. Eq. \ref{eq1} transforms to:
\begin{equation}\label{eq3}
\begin{split}
\sum_{j = 0}^{n_\frakP-1} &|F(\rho_{j+1}) - F(\rho_{j})|_{\k}  = \\
& \left( \sum_{j=0}^{n_\frakP-1} |F_1(p_{j + 1, 1}) - F_1(p_{j, 1})| \right)\ke + \left( \sum_{j=0}^{n_\frakP-1} |F_2(p_{j + 1, 2}) - F_2(p_{j, 2})| \right)\ked.
\end{split}
\end{equation}

\begin{remark}\label{rk1}
Eq. \ref{eq2} implies that the sums in Eq. \ref{eq1} are taken over the projections $\frakP_{\ke}$, $\frakP_{\ked}$, which are partitions of the real intervals $\left( [\alpha, \beta]_{\k} \right)_{\ke} = [a_1, b_1]$ and $\left( [\alpha, \beta]_{\k} \right)_{\ked} = [a_2, b_2]$ respectively. 
\end{remark}

Let us denote by $\calP([a_{j}, b_{j}])$ the collection of all partitions of the real interval $[a_{j}, b_{j}]$ for every $j \in \{1, 2\}$ and introduce the following sets
\[ \begin{split}
\sum_{\calP([a_{j}, b_{j}])}(F_{j}) = \left\{ \sum_{j=0}^{n_P-1} |F_j(p_{j + 1}) - F_j(p_{j})| \ |\ P \in \calP([a_{j}, b_{j}]) \right\}.
\end{split} \]

\begin{theorem}\label{pt1}
Let $F:[\alpha, \beta]_{\k} \rightarrow \K$ be a function with every component depending on the respective component of the variable. Then 
\[ \sum_{\calP([\alpha, \beta]_{\k})} (F) = \sum_{\calP([a_{1}, b_{1}])}(F_{1})\ke + \sum_{\calP([a_{2}, b_{2}])}(F_{2})\ked. \]
\end{theorem}
\begin{proof}
Suppose we are given a point of the set  $\sum_{\calP([\alpha, \beta]_{\k})} (F).$
By Remark \ref{rk1}, the projections $\frakP_{\ke}$ and $\frakP_{\ked}$ are partitions over $[a_1, b_1]$ and $[a_2, b_2]$ respectively and it takes the form in the Eq. \ref{eq3}.

Reciprocally, two partitions $P \in \calP([a_{1}, b_{1}])$ and $Q \in \calP([a_{2}, b_{2}])$ define a strong partition $\frakP$ (see Eq. \ref{mkStrongPart}). Partition $\frakP$ fulfill with $\frakP_{\ke} = P$ and $\frakP_{\ked} = Q$, even if the process generates $n_{\frakP} =n_{P} + n_{Q}$ points, where $n_{P}$ and $n_{Q}$ denote the cardinality of $P$ and $Q$, no additional elements in the sum are added, since in degenerated intervals $|F_1(p_{j + 1}) - F_1(p_{j})| = 0$ or $|F_2(p_{j + 1}) - F_2(p_{j})| = 0$, implying that
\[ \begin{split}
\left( \sum_{j=0}^{n_\frakP-1} \Delta_{\frakP, j}F_1 \right) & \ke + \left( \sum_{j=0}^{n_\frakP-1} \Delta_{\frakP, j}F_2 \right)\ked =  \\
& \left(\sum_{j=0}^{n_P-1} |F_1(p_{j + 1}) - F_1(p_{j})| \right)\ke + \left( \sum_{j=0}^{n_Q-1} |F_2(q_{j + 1}) - F_2(q_{j})| \right)\ked. 
\end{split} \]
\end{proof}
It is combination of Def. \ref{d4} with Thm. \ref{pt1}, that makes definition of hyperbolic-valued functions of bounded variation allowable.

\begin{corollary}\label{crBVC}
Let $F:[\alpha, \beta]_{\k} \rightarrow \K$ be a function given by Eq. \ref{eq3}. The function $F$ has hyperbolic bounded variation if and only if the idempotent component functions $F_1:[a_1, b_1] \rightarrow \R$ and $F_2:[a_2, b_2] \rightarrow \R$ are functions of real bounded variation.
\end{corollary}
On account of this result, the set of all points of discontinuity of a hyperbolic function of bounded variation is well defined, which is due to the fact that a real function of bounded variation only has jump discontinuities and therefore the set of all points of discontinuity is numerable, see \cite[Sec. 6.8]{a1974}.
\begin{lemma}
Let $F:[\alpha, \beta]_{\k} \rightarrow \K$ is  a hyperbolic-valued functions of bounded variation which every component only depends of the respective variable component. Then, the set of all points of discontinuity of $F$ consists of a numerable union of perpendicular line segments to idempotent axes.
\end{lemma}
\begin{proof}
The components $F_1$ and $F_2$ from $F$ are real functions of bounded variation so, there exist two set $\{x_{1, n}\}_{n \in \N} \subset [a_1, b_1]$ and $\{x_{2, n}\}_{n \in \N} \subset [a_2, b_2]$ do consist of all discontinuities for $F_1$ and $F_2$ respectively.

For every point $y \in [a_2, b_2]$ and $n \in \N$, the point $x_{1, n}\ke + y\ked$ is a point of discontinuity of $F$. Thus, the set of discontinuities points contains the union 
\[\bigcup_{n \in \N} x_{1, n}\ke + [a_2, b_2]\ked.\]

Similarly, the set of discontinuities points of $F$ contains the union
\[\bigcup_{n \in \N} [a_1, b_1]\ke + x_{2, n}\ked.\]

It follows that the union 
\[\frakD(F) := \left( \bigcup_{n \in \N} x_{1, n}\ke + [a_2, b_2]\ked \right) \cup \left( \bigcup_{n \in \N} [a_1, b_1]\ke + x_{2, n}\ked \right)\] 
contains all point of discontinuity of $F$, which is clear from the fact that if there exists $\xi = x\ke + y\ked$ a discontinuity point of $F$, then $x$ will be a discontinuity point of $F_1$ or $y$ a discontinuity point of $F_2$, which imply that $x \in \{x_{1, n}\}_{n \in \N}$ or $y \in \{x_{2, n}\}_{n \in \N}$.
\end{proof}

\begin{theorem}
The set of all discontinuity points of a hyperbolic-valued function of bounded variation whose components depend only of the respective component of the variable, has a zero Lebesgue measure in the Euclidean plane.
\end{theorem}
\begin{proof}
The idea of the proof lies in the fact that every line has zero measure in the Euclidean plane and numerable union of them is again zero measure.
\end{proof}

Note that similar result fails with  hyperbolic Lebesgue measure, defined in \cite{ks2017}, replacing Lebesgue measure. Since, the Lebesgue measure is identified with $\mu = \mu_{\R}\ke + \mu_{\R}\ked$, where $\mu_{\R}$ is the Lebesgue measure on the real line, the set $x_{1, n}\ke + [a_2, b_2]\ked$ has not zero measure for all $n \in \N$, which is due to the fact that $[a_2, b_2]$ is not a zero measure real set.

\section{Hyperbolic-valued Riemann-Stieltjes Integral}\label{scHRSI}
The diameter of a real partition $P$ is defined as the maximum into the set of all lengths of successive intervals generated by $P$,
\[\diam(P) = \max \{ \lambda([p_{j+1}, p_j]) \ |\ j \in \{0, ..., n-1\} \}. \]
It is clear that there exists a natural extension of the notion of diameter of a partition to the case of strong partitions if we replace maximum by supremum, but this is not beneficial in the effort to formulate a Riemann-Stieltjes integral.
\begin{definition}\label{d7}
Let $\frakP$ be a strong partition of $[\alpha, \beta]_{\k}$. The diameter of $\frakP$ is defined to be the hyperbolic number
\[\diam_{\k}(\frakP) = \diam(\frakP_{\ke})\ke + \diam(\frakP_{\ked})\ked. \]
\end{definition}
The relation between the concepts of diameter and strong partition yields a notion of Riemann-Stieltjes Integral. 
\begin{definition}\label{d6}
Let $F: [\alpha, \beta]_{\k} \rightarrow \K$ and $G: \K \rightarrow \K$ be two hyperbolic functions. A hyperbolic number $\calI$ is called the Riemann-Stieltjes integral of $F$ respect to $G$, if for every $\epsilon \in \K^{+}$ there exists a $\delta \in \K^{+}$ such that 
\[\left| S_{\k}(\frakP, F, G) -  \calI \right|_{\k} = \left| \sum_{j = 0}^{n_{\frakP} - 1} F(\gamma_{j}) \left| G(\rho_{j + 1}) - G(\rho_{j}) \right|_{\k} - \calI \right|_{\k} \prec \epsilon,\]
for any strong partition $\frakP \in \calP([\alpha, \beta]_{\k})$ that fulfill the property $\diam_{\k}(\frakP) \prec \delta$ and  whatever selection $\gamma_j \in [\rho_{j+1}, \rho_j]_{\k}$, with $j \in \{0, ..., n_{\frakP} - 1\}.$
\end{definition}
The quantity $S_{\k}(\frakP, F, G)$ is called the Riemann-Stieltjes sum. In addition, when such $\calI \in \K$ exists, it will be denoted by $\displaystyle \calI := \int_{\alpha}^{\beta} F d_{\k}G$.

If $F$ and $G$ in Def. \ref{d6} are assumed to have components that only depend of the respective component from the variable, then the Riemann-Stieltjes sum is analogue to Eq. \ref{eq3} and hence
\[\begin{split}
S_{\k}(\frakP, F, G) &= \left(\sum_{j = 0}^{n_{\frakP} - 1} F_{1}(y_{j, 1})\left| G_{1}(p_{j + 1, 1}) - G_1(p_{j, 1})\right| \right)\ke  \\
  &+ \left(\sum_{j = 0}^{n_{\frakP} - 1} F_{2}(y_{j, 2})\left| G_{2}(p_{j + 1, 2}) - G_2(p_{j, 2})\right| \right)\ked. 
\end{split}\]

Since $F_1$ and $F_2$ are real valued functions defined on the respective projections of $[\alpha, \beta]_{\k}$, likewise $G_1$ and $G_2$ are real valued functions defined on the whole real line, hence the Riemann-Stieltjes sum over the hyperbolic plane is the idempotent sum of two classic Riemann-Stieltjes sums on the partition generated by projections of $\frakP$: 
\[S_{\k}(\frakP, F, G) = S(\frakP_{\ke}, F_1, G_1)\ke + S(\frakP_{\ked}, F_2, G_2)\ked.\]
This construction generalized that of \cite[Sec. 7.3]{a1974}.

From now on all functions will be considered with components depending only of the respective component from the variable.

Component-wise operation shows that hyperbolic Riemann-Stieltjes integral can be computed using classical Riemann-Stieltjes integration over the real line.
\begin{equation}\label{eq-rsk}
\int_{\alpha}^{\beta} F d_{\k}G = \left(\int_{a_1}^{b_1} F_1 dG_1\right)\ke + \left( \int_{a_2}^{b_2} F_2 dG_2 \right)\ked,
\end{equation}

The preceding observations, leads to the following result:
\begin{theorem}\label{tmERS}
A hyperbolic function $F:[\alpha, \beta]_{\k} \rightarrow \K$ is hyperbolic Riemann-Stieltjes integrable with respect to a hyperbolic function $G: \K \rightarrow \K$, if and only if the components $F_1:[a_1, b_1] \rightarrow \R$ and $F_{2}:[a_2, b_2] \rightarrow \R$ are real Riemann-Stieltjes integrable functions respect to $G_{1}:\R \rightarrow \R$ and $G_{2}: \R \rightarrow \R$.
\end{theorem}

With the identity function $Id_{\k}: \K \rightarrow \K$, the Riemann integral defined in \cite[Ch. IV]{vd1935} of a function $F$ is recovered when it is computed the hyperbolic Riemann-Stieltjes integral respect to $Id_{\k}$. Also, the last is equivalent to the Lebesgue integral introduced in  \cite[Sec. 3]{ks2017} with the hyperbolic-valued Lebesgue measure in $\K$.

Results in \cite{vd1935} requires non-self-intersecting continuous loop (Jordan curve). The straight line that joins the two extreme points of a hyperbolic interval gives a loop of this kind. Therefore, for every strong partition $\frakP$ the union of lines that join every sub-interval $[\rho_{j+1}, \rho_{j}]_{\k}$, where $j \in \{0, ..., n-1\}$, is a Jordan loop and
\[\int_{\alpha}^{\beta} F d_{\k}Id_{\k}  = \int_{\alpha}^{\beta} F d_{\k}\xi = \left(\int_{a_1}^{b_1} F_1 dx_1 \right) \ke + \left( \int_{a_2}^{b_2} F_2 dx_2 \right) \ked. \]

For Lebesgue measure $\mu_{\k} := \mu_{\R}\ke + \mu_{\R}\ked$, it is necessary to take into account that the Lebesgue integral is defined in a component-wise way and that the real Lebesgue integral restricted to a closed interval reduces to the Riemann integral,
\[\int_{\alpha}^{\beta} F d_{\k}\xi = \int_{[\alpha, \beta]_{\k}} F d\mu_{\k} = \left(\int_{[a_1, b_1]}F_{1}d\mu_{\R} \right)\ke + \left( \int_{[a_2, b_2]} F_2d\mu_{\R} \right)\ked\]

A primary study of a Riemann integration of hyperbolic-valued functions was presented in \cite{el2022}. In this paper, are established some basic properties and results on the introduced notion of integration. Def. 5.2 naturally generalizes and strengthens that Riemann integral. 

Let us mention an important property of the hyperbolic Riemann-Stieltjes integral, when the integrator is an holomorphic function.
\begin{theorem}\label{thRSR}
Let $G: \K \rightarrow \K$ be an holomorphic and continuously differentiable function, $F:[\alpha, \beta] \rightarrow \K$ a hyperbolic Riemann-Stieltjes integrable function with respect to $G$. Then
\[\int_{\alpha}^{\beta} F d_{\k}G =  \int_{\alpha}^{\beta} FG' d_{\k} \xi. \]
\end{theorem}
\begin{proof}
Since $G$ is holomorphic hyperbolic function, the derivative $G'(\xi) = G_1'(x)\ke + G_2'(y)\ked$ exists, for $\xi \in \K$. By the continuously differentiability of $G$ its idempotent components have continuous derivatives of any order. Therefore $G_1$ and $G_2$ are functions of bounded variation and \cite[Thm. 7.8]{a1974} makes easy to see that
\[ \int_{a_1}^{b_1} F_1 dG_1 = \int_{a_1}^{b_1} F_1G_1' dx_1\quad \text{and} \quad \int_{a_2}^{b_2} F_2 dG_2 = \int_{a_2}^{b_2} F_2G_2' dx_2.\]
Combining these equalities the result is obtained.
\end{proof}
\begin{remark}
The assertion of Thm. \ref{thRSR} does not follow if requirement on $G$ to be a continuously differentiable function is omitted from the hypotheses. Unlike happen in the complex analysis context, hyperbolic holomorphic functions does not have derivatives of all orders, see \cite{vd1935, mr1998, ks2005}.
\end{remark}

Thm. \ref{thRSR} establishes a direct relation between Riemann-Stieltjes and Riemann integrals, when integrability of $F$ respect to $G$ is required. So, it is convenient to see under what conditions this integrability holds. 
\begin{theorem}
Suppose that $F: [\alpha, \beta]_{\k} \rightarrow \K$ is a continuous hyperbolic function and $G: \K \rightarrow \K$ is a hyperbolic function of bounded variation. Then $F$ is Riemann-Stieltjes integrable with respect to $G$.
\end{theorem}
\begin{proof}
The components $F_1$ and $F_2$ of a continuous hyperbolic function $F$, are real continuous functions and the Cor. \ref{crBVC} shows that $G_1$ and $G_2$ are real functions of bounded variation. Therefore, the integrals 
\[ \int_{a_1}^{b_1}F_1 dG_1 \quad and \quad \int_{a_2}^{b_2}F_2 dG_2\]
exist, which is clear from \cite[Thm. 7.27]{a1974}. Finally, by Thm. \ref{tmERS}, the result is obtained.
\end{proof}

\subsection*{Statements and Declarations}
This is part of the first author's Ph.D. thesis, written under the supervision of the second author at Instituto Polit\'ecnico Nacional.
\subsubsection*{Author Contributions} The authors contributed equally in this paper and typed, read, and approved the final
form of the manuscript.
\subsubsection*{Funding} Partial financial support was received from Instituto Politécnico Nacional (grant number
SIP20211188) and Postgraduate Study Fellowship of the Consejo Nacional deCiencia y Tecnología (CONACYT)
(grant number 744134).
\subsubsection*{Data and Material Availability} No data were used to support this study.
\subsubsection*{Code Availability} Not applicable.
\subsubsection*{Conflict of interest} The authors declare that they have no conflict of interest regarding the work reported in this paper.

\end{document}